\documentclass[11pt,reqno,letterpaper]{amsart}

\usepackage{amssymb}
\usepackage{amsmath}
\usepackage{amsthm}
\usepackage{bbm}
\usepackage{doi}
\usepackage[shortlabels]{enumitem}
\usepackage[amsstyle]{cases}

\usepackage{bookmark}
\usepackage{hyperref}
\hypersetup{pdfstartview={FitH}}

\numberwithin{equation}{section}

\newtheorem{theorem}{Theorem}
\newtheorem{lemma}[theorem]{Lemma}
\newtheorem{proposition}[theorem]{Proposition}
\newtheorem{corollary}[theorem]{Corollary}

\theoremstyle{definition}
\newtheorem{remark}[theorem]{Remark}

\renewcommand{\leq}{\leqslant}
\renewcommand{\geq}{\geqslant}

\newcommand{\R}{\mathbb{R}}
\newcommand{\N}{\mathbb{N}}
\newcommand{\Z}{\mathbb{Z}}

\begin{document}

\title{Sharp estimates for Gowers norms on discrete cubes}

\author[A. Beker]{Adrian Beker}
\address{A. B., Department of Mathematics, Faculty of Science, University of Zagreb, Bijeni\v{c}ka cesta 30, 10000 Zagreb, Croatia}
\email{adrian.beker@math.hr}

\author[T. Crmari\'{c}]{Ton\'{c}i Crmari\'{c}}
\address{T. C., Department of Mathematics, Faculty of Science, University of Split, Ru\dj{}era Bo\v{s}kovi\'{c}a 33, 21000 Split, Croatia}
\email{tcrmaric@pmfst.hr}

\author[V. Kova\v{c}]{Vjekoslav Kova\v{c}}
\address{V. K., Department of Mathematics, Faculty of Science, University of Zagreb, Bijeni\v{c}ka cesta 30, 10000 Zagreb, Croatia}
\email{vjekovac@math.hr}

\subjclass[2020]{Primary 05D05; 
Secondary 11B30, 
94A17} 

\keywords{additive energy, binary cube, Lebesgue norm, Shannon entropy, asymptotic formula}

\begin{abstract}
We study optimal dimensionless inequalities 
\[ \|f\|_{\textup{U}^k} \leq \|f\|_{\ell^{p_{k,n}}} \] 
that hold for all functions $f\colon\mathbb{Z}^d\to\mathbb{C}$ supported in $\{0,1,\ldots,n-1\}^d$ and estimates
\[ \|\mathbbm{1}_A\|_{\textup{U}^k}^{2^k}\leq |A|^{t_{k,n}} \]
that hold for all subsets $A$ of the same discrete cubes. A general theory, analogous to the work of de Dios Pont, Greenfeld, Ivanisvili, and Madrid, is developed to show that the critical exponents are related by $p_{k,n} t_{k,n} = 2^k$. This is used to prove the three main results of the paper:
\begin{itemize}
\item an explicit formula for $t_{k,2}$, which generalizes a theorem by Kane and Tao,
\item two-sided asymptotic estimates for $t_{k,n}$ as $n\to\infty$ for a fixed $k\geq2$, which generalize a theorem by Shao, and 
\item a precise asymptotic formula for $t_{k,n}$ as $k\to\infty$ for a fixed $n\geq2$.
\end{itemize}
\end{abstract}

\maketitle


\section{Introduction}
Motivation for this article is two-fold. From the combinatorial side, we continue the recently started line of investigation of upper bounds on generalized additive energies of subsets $A$ of the discrete cube 
\begin{equation}\label{eq:thecube}
\{0,1,2,\ldots,n-1\}^d. 
\end{equation}
From the analytical side, we take omnipresent classical inequalities between the Gowers norms and the Lebesgue norms for functions $f\colon\Z^d\to\mathbb{C}$ and sharpen them when we additionally assume that $f$ is supported in the cube \eqref{eq:thecube}.

The \emph{additive energy} 
\begin{equation}\label{eq:addenergy}
E_2(A) := \big| \{ (a_1,a_2,a_3,a_4)\in A^4 \,:\, a_1-a_2 = a_3-a_4 \} \big|
\end{equation}
of a finite set $A\subset\Z^d$ appears naturally in additive combinatorics; see the book by Tao and Vu \cite[Section 2.3]{TV06}.
For sets $A\subseteq\{0,1\}^d\subset\Z^d$, Kane and Tao \cite[Theorem 7]{KT17} proved the optimal energy bound in terms of the set size,
\[ E_2(A) \leq |A|^{\log_2 6}. \]
The exponent $\log_2 6$ is sharp since the equality holds when $A$ is the whole binary cube $\{0,1\}^d$.
Three generalizations of $E_2$ also appear naturally in applications and they were also systematically studied on their own by Schoen and Shkredov \cite{SS13} and Shkredov \cite{Shk14}.
De Dios Pont, Greenfeld, Ivanisvili, and Madrid \cite{DGIM21} revisited two of them and called them the \emph{higher energies}
\[ \widetilde{E}_k(A) := \big| \{ (a_1,a_2,\ldots,a_{2k-1},a_{2k})\in A^{2k} \,:\, a_1-a_2 = a_3-a_4 = \cdots = a_{2k-1}-a_{2k} \} \big| \]
and the \emph{$k$-additive energies}
\[ E_k(A) := \big| \{ (a_1,a_2,\ldots,a_{2k-1},a_{2k})\in A^{2k} \,:\, a_1+\cdots+a_k = a_{k+1}+\cdots+a_{2k} \} \big|. \]
Here $k\geq 2$ is an integer and $A$ is again a finite subset of $\Z^d$. Note that $\widetilde{E}_2=E_2$ is the usual additive energy \eqref{eq:addenergy}. The authors of \cite{DGIM21} then also restricted their attention to the sets $A\subseteq\{0,1\}^d$.
Following an inductive scheme similar to the one in \cite[Section 2]{KT17}, they proved sharp inequalities for $\widetilde{E}_k$,
\[ \widetilde{E}_k(A) \leq |A|^{\log_2 (2^k+2)}; \]
see \cite[Theorem 1]{DGIM21}.
The energies $E_k$ satisfy similarly looking sharp estimates
\[ E_k(A) \leq |A|^{\log_2 \binom{2k}{k}}, \]
but the proof was more technical: it was done for $k\leq 100$ in \cite[Section 3]{DGIM21}, while one of the authors of the present paper established an ingredient needed for general $k$ in \cite{Kov23}.
A third possible quantity that generalizes \eqref{eq:addenergy} counts (possibly degenerate) parallelotopes with vertices in $A$, with a lot of overcounting. It is defined as
\begin{align} 
P_k(A) := \big| \big\{ (a,h_1,\ldots,h_k)\in(\Z^d)^{k+1} \,:\ & a+\epsilon_1 h_1+\cdots+\epsilon_k h_k \in A \nonumber \\
& \text{for every } (\epsilon_1,\ldots,\epsilon_k)\in\{0,1\}^k \big\} \big| \label{eq:Pkdef}
\end{align}
for any $k\geq2$. It has already been studied by Shkredov \cite{Shk14}, who also discussed numerous relations between all three kinds of generalized additive energies.
It is precisely the generalized energies $P_k$ that are the topic of interest in this paper.

For a general complex function $f$ defined on an additively written (discrete) abelian group and a positive integer $k$ one defines the \emph{Gowers uniformity norm} as
\[ \|f\|_{\textup{U}^k} := \bigg( \sum_{a,h_1,\ldots,h_k} \prod_{(\epsilon_1,\ldots,\epsilon_k)\in\{0,1\}^k} \mathcal{C}^{\epsilon_1+\cdots+\epsilon_k} f(a+\epsilon_1 h_1+\cdots+\epsilon_k h_k) \bigg)^{1/2^k}, \]
where $\mathcal{C}\colon z\mapsto \bar{z}$ denotes the operator of complex conjugation, so that $\mathcal{C}^2$ is the identity.
This is, in fact, a norm for $k\geq2$, defined on the set of functions for which the above expression is finite; see Section \ref{sec:notation}. 
These norms were introduced by Gowers \cite{Gow98,Gow01} in his quantitative proof of Szemer\'{e}di's theorem.
If $f$ is specialized to be the indicator function $\mathbbm{1}_A$ of a finite set $A\subset\Z^d$, then we recover precisely the quantity $P_k$ from \eqref{eq:Pkdef}, namely
\[ P_k(A) = \|\mathbbm{1}_A\|_{\textup{U}^k}^{2^k}. \]
For a finitely supported function $f\colon\Z^d\to\mathbb{C}$ the inequality
\begin{equation}\label{eq:GowersLp}
\|f\|_{\textup{U}^k} \leq \|f\|_{\ell^p}
\end{equation}
is well-known to hold for every exponent
\begin{equation}\label{eq:critrange}
p\leq \frac{2^k}{k+1};
\end{equation}
see \cite{ET12}. Estimate \eqref{eq:GowersLp}, no matter how simple, has found versatile important applications, for instance in the study of generalized multiplicative sequences \cite[Section 2]{FK19} and automatic sequences \cite[Proposition 2.2]{BKM23}, or in geometric measure theory of large sets \cite[Section 6]{DK22}.
The number $2^k/(k+1)$ was called the critical exponent in \cite[Subsection 1.1]{ET12} and it cannot be increased. Namely, it is easy to see that 
\begin{equation}\label{eq:comparable}
\|\mathbbm{1}_{\{0,1,\ldots,n-1\}}\|_{\textup{U}^k} \text{ is comparable to } n^{(k+1)/2^k}
\end{equation}
up to an unimportant multiplicative constant, while
\[ \|\mathbbm{1}_{\{0,1,\ldots,n-1\}}\|_{\ell^p} = n^{1/p}. \]
If we had \eqref{eq:GowersLp} with some $p>2^k/(k+1)$, we would arrive at a contradiction in the limit as $n\to\infty$.

One naturally wonders if \eqref{eq:critrange} can still be improved if we additionally assume that $f$ is supported in the binary cube. Also, one might attempt to control generalized additive energies $P_k$ of $A\subseteq\{0,1\}^d$ by optimal powers of the size of $A$. Both of these problems are resolved by the first result of this paper.

\begin{theorem}\label{thm:binary}
Let $d\geq0$ and $k\geq2$ be integers. For every function $f\colon\Z^d\to\mathbb{C}$ supported in $\{0,1\}^d$ the inequality \eqref{eq:GowersLp} holds for every exponent
\[ 0 < p \leq \frac{2^k}{\log_2 (2k+2)}. \]
In particular, by taking the largest possible $p$ and $f=\mathbbm{1}_A$, we get that
\begin{equation}\label{eq:Pkest}
P_k(A) \leq |A|^{\log_2 (2k+2)}
\end{equation}
holds for every set $A\subseteq\{0,1\}^d$.
\end{theorem}

Both the stated range of $p$ and inequality \eqref{eq:Pkest} are sharp since the equality holds for $A=\{0,1\}^d$. Note that $\Z^0$ is interpreted as a group consisting only of the neutral element for the addition.

\smallskip
Estimates for $E_2$ on larger discrete cubes \eqref{eq:thecube} were studied by de Dios Pont, Greenfeld, Ivanisvili, and Madrid \cite[Subsection 1.3]{DGIM21} and Shao \cite{Sha24}; also see a more general setting by Hegyv\'{a}ri \cite{Heg23}.
In this paper we also want to study nontrivial inequalities for $\|f\|_{\textup{U}^k}$ and $P_k(A)$, i.e., \eqref{eq:GowersLp} and
\begin{equation}\label{eq:GowersA}
P_k(A) \leq |A|^t,
\end{equation}
on discrete cubes.
The first step in these considerations is the following general proposition, which reduces higher-dimensional estimates to an optimization problem in finitely many variables. It will be formulated and proven in the spirit of the results from \cite[Section 4]{DGIM21}. 

\begin{proposition}\label{prop:product}
Let $k,n\geq2$ be integers and let $p,t>0$ be real numbers such that $pt=2^k$. The following are equivalent.
\begin{enumerate}[$(1)$]
\item\label{it1}
Inequality \eqref{eq:GowersLp} holds for every function $f\colon\Z\to[0,\infty)$ supported in the set $\{0,1,2,\ldots,n-1\}$.
\item\label{it2}
Inequality \eqref{eq:GowersLp} holds for every integer $d\geq0$ and every function $f\colon\Z^d\to\mathbb{C}$ supported in the discrete cube \eqref{eq:thecube}.
\item\label{it3}
Inequality \eqref{eq:GowersA} holds for every integer $d\geq0$ and every subset $A$ of the discrete cube \eqref{eq:thecube}.
\item\label{it4}
There exists a constant $C\in(0,\infty)$ such that $P_k(A) \leq C |A|^t$ holds for every integer $d\geq0$ and every subset $A$ of \eqref{eq:thecube}.
\end{enumerate}
\end{proposition}

All ingredients for the proof of Proposition \ref{prop:product} will be borrowed from \cite{DGIM21}, even though the authors there rather discussed norms of multiple convolutions,
\[ \big\|\underbrace{f\ast f\ast \cdots \ast f}_{k\text{ times}}\big\|_{\ell^2}^2 \]
and the related $k$-additive energies $E_k$.
Another difference is that the proofs in \cite{DGIM21} use the Fourier transform, while we avoid the need to look at the functions on the Fourier side at all. This is important here and in the later text, since the higher Gowers norms $\|f\|_{\textup{U}^k}$, $k\geq3$, are not expressible merely in terms of the size of the Fourier transform of $f$. 

\smallskip
Let $t_{k,n}$ denote the smallest number $t>0$ such that \eqref{eq:GowersA} holds for every positive integer $d$ and every subset $A$ of \eqref{eq:thecube}.
The smallest such number really exists by Proposition \ref{prop:product}, since part \ref{it1} is structurally just an inequality for $n$ non-negative numbers and it can be normalized to yield a compact condition.
By the same proposition we then also know that
\begin{equation}\label{eq:pknvstkn}
p_{k,n}t_{k,n}=2^k,
\end{equation}
where $p_{k,n}$ is the largest $p>0$ such that \eqref{eq:GowersLp} holds in any dimension $d$ and for every function $f$ supported in \eqref{eq:thecube}.

Theorem \ref{thm:binary} can now be reformulated simply as
\[ t_{k,2} = \log_2 (2k+2) \]
for $k\geq2$.
It is not very likely that explicit expressions for $t_{k,n}$ can be computed when $n\geq 3$.
Namely, a similar computation to the one in \cite[Subsection 4.3]{DGIM21} shows that the exponent $t_{k,3}$ associated with the ternary cube $\{0,1,2\}^d$ is the smallest $t>0$ such that
\begin{align}
\max_{\substack{x,y,z\in[0,\infty)\\ x+y+z=1}} \bigl( & x^t + y^t + z^t + 2k x^{t/2} y^{t/2} + 2k y^{t/2} z^{t/2} \nonumber \\[-3mm]
& + 2k x^{t/2} z^{t/2} + 2k(k-1) x^{t/4} y^{t/2} z^{t/4} \bigr) = 1. \label{eq:tern}
\end{align}
Already the number $t_{2,3}$ does not seem to be a ``nice'' explicit number. As remarked in \cite[Subsection 4.3]{DGIM21} its trivial bounds are
\[ 2.68014\ldots = \log_3 19 \leq t_{2,3} \leq 3, \]
where it is also shown that
\[ t_{2,3} \geq 2 \log_2 2.5664 = 2.71949\ldots. \]
Numerical computation in Mathematica \cite{Mathematica} based on \eqref{eq:tern} and formula \eqref{eq:pknvstkn} give 
\begin{align*}
t_{2,3} & = 2.7207109973\ldots, \\
p_{2,3} & = 1.4702039297\ldots.
\end{align*}
Inverse Symbolic Calculator \cite{ISC} does not recognize these two as any of the numbers appearing naturally elsewhere. 
Still, as a consequence of a general Theorem \ref{thm:general2} below, we will have a quite precise asymptotics of $t_{k,3}$ for large $k$, namely
\begin{equation}\label{eq:tk3asy}
t_{k,3} = \frac{4}{3}\log_2 k + \frac{2}{3} + o^{k\to\infty}(1).
\end{equation}

\smallskip
For the previous reason, the best one can hope for general $k\geq2$ and $n\geq3$ is to study the asymptotic behavior of $t_{k,n}$ as either $n\to\infty$ or $k\to\infty$. Let us first fix $k$ and study the asymptotics in $n$.
As we have already mentioned, the numbers $t_{2,n}$ were studied before, since $P_2$ coincides with the ordinary additive energy \eqref{eq:addenergy}.
The bound
\[ 3 - \frac{\log_2 3 - 1}{\log_2 n} \leq t_{2,n} \leq 3 \]
from \cite[Proposition 7]{DGIM21} was improved by Shao \cite{Sha24} to
\begin{equation}\label{eq:Shao}
3 - \big(1 + o^{n\to\infty}(1)\big) \frac{3\log_2 3 - 4}{2\log_2 n} \leq t_{2,n} \leq 3 - \frac{c}{\log_2 n}
\end{equation} 
for some constant $c>0$.
Shao also conjectured that the leftmost expression in \eqref{eq:Shao} is the correct asymptotics of $t_{2,n}$ as $n\to\infty$.
Here we generalize Shao's results to $t_{k,n}$ for $k\geq3$.

\begin{theorem}\label{thm:general}
There exists a constant $c\in(0,\infty)$ such that for every integer $k\geq 2$ we have
\[ k+1 - \big(1 + o_{k}^{n\to\infty}(1)\big) \frac{(k+1)\log_2(k+1)-2k}{2\log_2 n} \leq t_{k,n} \leq k+1 - \frac{c}{\log_2 n}. \]
\end{theorem}

Note that the lower bound from Theorem \ref{thm:general} specializes to the one in \eqref{eq:Shao} when $k=2$ and one can again speculate that it is optimal. The constant $c$ in the upper bound is the same one from \eqref{eq:Shao}, as we are, in fact, applying Shao's highly nontrivial bound as a black box. One could, in fact, repeat the proof from \cite{Sha24} to obtain constants $c_k$ that strictly increase with $k$. However, we found it difficult to track down the actual growth and it is unlikely that it could match the growth of $(1/2)(k+1)\log_2(k+1)-k$ from the lower bound.

\smallskip
In this paper we also have the opportunity to fix $n\geq 2$ and study the asymptotics of $t_{k,n}$ as $k\to\infty$.
Here is where we can give a quite definite result; we find it the most substantial contribution of this paper. 
Let
\begin{equation}\label{eq:binomentr}
H_m := - \sum_{j=0}^{m} \frac{\binom{m}{j}}{2^m} \log_2 \frac{\binom{m}{j}}{2^m} 
\end{equation}
denote the entropy of the symmetric binomial distribution $B(m,1/2)$; see Section \ref{sec:notation} for the more general definition.

\begin{theorem}\label{thm:general2}
For every integer $n\geq 2$ we have
\[ t_{k,n} = \frac{(n-1)\log_2 (2k) - \log_2 (n-1)!}{H_{n-1}} + o_{n}^{k\to\infty}(1). \]
\end{theorem}

As a consequence, $t_{k,n}$ is asymptotically equal to
\[ \frac{n-1}{H_{n-1}}\log_2 k \]
as $k\to\infty$.
Several values of the leading coefficient $(n-1)/H_{n-1}$ are listed in Table \ref{tab:values}.
We also mention that it is asymptotically equal to $2n/\log_2 n$ as $n\to\infty$; this is a well-known fact, which also follows from sharper estimates \eqref{eq:sharpHmest} in Section \ref{sec:proofofkthm}. 
In particular, the constant $t_{k,3}$ relevant for the ternary cube is asymptotically equal to $(4/3)\log_2 k$ and Theorem \ref{thm:general2} also gives the more precise formula \eqref{eq:tk3asy}.
Somewhat surprisingly, the proof of Theorem \ref{thm:general2} will require subtle estimates for the Shannon entropy of certain independent sums of discrete random variables, which will be discussed in Subsection \ref{subsec:entropy}.

\begin{table}[t]
\begin{center}
\begin{tabular}{|r||l|l|}
\hline
$n$ & $(n-1)/H_{n-1}$ & numerical value \\ 
\hline\hline
$2$ & $1$ & $1$ \\  
\hline
$3$ & $4/3$ & $1.3333333333\ldots$ \\
\hline
$4$ & $4/(4 - \log_2 3)$ & $1.6562889815\ldots$ \\
\hline
$5$ & $32/(21 - 3\log_2 3)$ & $1.9698232317\ldots$ \\
\hline
$6$ & $16/(14 - 3\log_2 5)$ & $2.2745961522\ldots$ \\
\hline 
\end{tabular}
\end{center}
\caption{Several values of $(n-1)/H_{n-1}$.}
\label{tab:values}
\end{table}


\section{Notation and preliminaries}
\label{sec:notation}
Cardinality of a finite set $S$ is simply written as $|S|$, just as we already did in the introduction.
We write $\N$ for the set of positive integers, $\{1,2,3,\ldots\}$. Other number sets, namely $\Z$, $\R$, and $\mathbb{C}$, have their usual meanings.
Logarithms with different bases will be used in the text; by far the most commonly used will be $2$. The logarithm with base $e$ will be denoted $\ln$.

If $U$ is a fixed (universal) set, then we write $\mathcal{P}(U)$ for the collection of all of its subsets. We say that two sets $A, B \in\mathcal{P}(U)$ are \emph{comparable} if one of them is a subset of the other. Given a family of subsets $\mathcal{A} \subseteq \mathcal{P}(U)$, we say that $\mathcal{A}$ is a \emph{chain} if any two sets in $\mathcal{A}$ are comparable and that $\mathcal{A}$ is an \emph{antichain} if no two sets in $\mathcal{A}$ are comparable.

Suppose that
\[ F(m,a,b,\ldots)\in\R \quad\text{and}\quad G(m,a,b,\ldots)\in(0,\infty) \]
are quantities defined for sufficiently large positive integers $m$ that possibly also depend on certain parameters $a,b,\ldots$.
We write
\[ F(m,a,b,\ldots) = O_{a,b,\ldots}^{m\to\infty}\big(G(m,a,b,\ldots)\big) \]
if
\[ \limsup_{m\to\infty} \frac{|F(m,a,b,\ldots)|}{G(m,a,b,\ldots)} < \infty \]
for every possible $a,b,\ldots$, while 
\[ F(m,a,b,\ldots) = o_{a,b,\ldots}^{m\to\infty}\big(G(m,a,b,\ldots)\big) \]
means that
\[ \lim_{m\to\infty} \frac{F(m,a,b,\ldots)}{G(m,a,b,\ldots)} = 0 \]
holds for every fixed choice of $a,b,\ldots$.
We also simply write 
\[ O_{a,b,\ldots}^{m\to\infty}\big(G(m,a,b,\ldots)\big) \quad\text{and}\quad o_{a,b,\ldots}^{m\to\infty}\big(G(m,a,b,\ldots)\big) \]
in place of any particular expression $F(m,a,b,\ldots)$ with the above property.
It is understood that these definitions are uniform over any other objects that do not appear in the subscript of the $O$ or $o$ notation.
Next, we say that sequences $(F(m))_m$ and $(G(m))_m$ are asymptitocally equal if
\[ \lim_{m\to\infty} \frac{F(m)}{G(m)} = 1, \]
i.e.,
\[ F(m) = (1+o^{m\to\infty}(1)) \,G(m). \]

The \emph{$\ell^p$ norm} of a function $f\colon\mathbb{Z}^d\to\mathbb{C}$ is defined as
\[ \|f\|_{\ell^p} := \Big(\sum_{x\in\mathbb{Z}^d} |f(x)|^p \Big)^{1/p} \]
for $p\in[1,\infty)$, while the \emph{$\ell^\infty$ norm} is simply
\[ \|f\|_{\ell^\infty} := \sup_{x\in\mathbb{Z}^d} |f(x)|. \]
\emph{Convolution} of functions $f,g\colon\Z^d\to\mathbb{C}$ is another complex function on $\Z^d$ defined as
\[ (f\ast g)(x) := \sum_{y\in\Z^d} f(x-y) g(y) \]
for $x\in\Z^d$.
\emph{Young's inequality for convolution} now reads
\begin{equation}\label{eq:Youngineq}
\|f\ast g\|_{\ell^r} \leq \|f\|_{\ell^p} \|g\|_{\ell^q},
\end{equation}
which holds when $p,q,r\in[1,\infty]$ are such that $1/p+1/q=1+1/r$.
We also define $\widetilde{f}$ to be the reflection of $f$, namely
\[ \widetilde{f}(x) := f(-x) \]
for every $x\in\Z^d$, so that $\|\widetilde{f}\|_{\ell^p}=\|f\|_{\ell^p}$ for all $p\in[1,\infty]$.

The \emph{Gowers ``inner product''} and its properties will play a crucial role in the proofs given in Sections \ref{sec:binary} and \ref{sec:product}. It is defined to be
\begin{align*} 
& \big\langle (f_{\epsilon_1,\ldots,\epsilon_k})_{(\epsilon_1,\ldots,\epsilon_k)\in\{0,1\}^k} \big\rangle_{\textup{U}^k} \\
& := \sum_{a,h_1,\ldots,h_k} \prod_{(\epsilon_1,\ldots,\epsilon_k)\in\{0,1\}^k} \mathcal{C}^{\epsilon_1+\cdots+\epsilon_k} f_{\epsilon_1,\ldots,\epsilon_k}(a+\epsilon_1 h_1+\cdots+\epsilon_k h_k),
\end{align*}
where $(f_{\epsilon_1,\ldots,\epsilon_k})$ is now a $2^k$-tuple of complex functions on $\Z^d$ such that the above multiple series converges absolutely.
It satisfies the so-called \emph{Gowers--Cauchy--Schwarz inequality}:
\[ \big| \big\langle (f_{\epsilon_1,\ldots,\epsilon_k})_{(\epsilon_1,\ldots,\epsilon_k)\in\{0,1\}^k} \big\rangle_{\textup{U}^k} \big|
\leq \prod_{(\epsilon_1,\ldots,\epsilon_k)\in\{0,1\}^k} \|f_{\epsilon_1,\ldots,\epsilon_k}\|_{\textup{U}^k}; \]
see \cite[Lemma 3.8]{Gow01}.
One if its consequences is that the Gowers norms satisfy the triangle inequality,
\[ \|f_1 + f_2\|_{\textup{U}^k} \leq \|f_1\|_{\textup{U}^k} + \|f_2\|_{\textup{U}^k}; \]
see \cite[Lemma 3.9]{Gow01}.

It will also be convenient to recall the classical inequality of Karamata \cite{Kar32}.
Suppose that we are given two $N$-tuples 
\[ \mathbf{x} = (x_1,x_2,\ldots,x_N),\quad \mathbf{y} = (y_1,y_2,\ldots,y_N) \]
of numbers from an interval $I\subseteq\R$ sorted in the descending order, i.e., 
\[ x_1\geq x_2\geq \cdots \geq x_N, \quad y_1\geq y_2\geq \cdots \geq y_N. \]
We say that $\mathbf{x}$ \emph{majorizes} $\mathbf{y}$ if
\[ \sum_{j=1}^{N} x_j = \sum_{j=1}^{N} y_j \]
and
\[ \sum_{j=1}^{n} x_j \geq \sum_{j=1}^{n} y_j \]
for all $1\leq n\leq N-1$.
Karamata proved that, for a strictly convex function $\psi\colon I\to\R$, we have
\[ \sum_{j=1}^{N} \psi(x_j) \geq \sum_{j=1}^{N} \psi(y_j) \]
with equality only if $\mathbf{x}=\mathbf{y}$; see the details in \cite[Theorem 1, Note 2]{kadelburg-djukic-lukic-matic}.
The inequality is reversed if $\psi$ is strictly concave.

Given an integer-valued random variable $X$ on a probability space $(\Omega,\mathcal{F},\mathbb{P})$ we define its \emph{probability mass function} to be the function
\[ p_X \colon \mathbb{Z} \to [0,1], \quad p_X(z) := \mathbb{P}(X=z).\]
We will only work with distributions that are supported on finitely many integers. In this case the \emph{decreasing rearrangement} of $p_X$ is defined to be the decreasing (i.e., non-increasing) function $p_X^{\downarrow} \colon \mathbb{N} \to [0,1]$ taking the same multiset of values as $p_X$.

The notion of the Shannon entropy will be needed in Section \ref{sec:proofofkthm}.
If a random variable $X$ has distribution supported on a finite subset of $\Z$, then its \emph{Shannon entropy} (or simply just \emph{entropy}) is defined to be the number
\begin{equation}\label{eq:entrdef}
\textup{H}(X) := -\sum_{z\in\Z} p_X(z) \log_2 p_X(z).
\end{equation}
Here we interpret $0\log_2 0$ as $0$. 
In the particular case of a random variable $X$ distributed
\[ X \sim \begin{pmatrix}
0 & 1 & 2 & \cdots & n-1 \\
q_0 & q_1 & q_2 & \cdots & q_{n-1}
\end{pmatrix} \]
its entropy will sometimes also be written as
\[ \textup{H}(q_0,\ldots,q_{n-1}). \]
Recall that, in accordance with formula \eqref{eq:binomentr}, the entropy of the symmetric binomial distribution with $m$ trials, namely $B(m,1/2)$, is written simply as $H_{m}$.


\section{Proof of Proposition \ref{prop:product}}
\label{sec:product}
Recall that $p$ and $t$ are now arbitrary positive numbers related by
\[ pt=2^k. \]

\emph{Proof of \ref{it1}$\implies$\ref{it2}}.
The assumption \ref{it1} is that for every $g\colon\Z\to[0,\infty)$ supported in $\{0,1,\ldots,n-1\}$ one has
\begin{equation}\label{eq:auxpower3}
\sum_{b,l_1,\ldots,l_k\in\Z} \prod_{(\epsilon_1,\ldots,\epsilon_k)\in\{0,1\}^k} g(b+\epsilon_1 l_1+\cdots+\epsilon_k l_k)
\leq \Big( \sum_{b=0}^{n-1} g(b)^p \Big)^{2^k/p}.
\end{equation}
We prove the claim \ref{it2} by induction on $d$. The base case $d=0$ is trivial again, since all functions on $\Z^0$ are constants.
Now take $d\in\N$ and a complex function $f$ on $\Z^d$ supported in \eqref{eq:thecube}. For each $b\in\Z$ define
\[ f_b\colon\Z^{d-1}\to\mathbb{C}, \quad f_b(a) := f(a,b) \]
for $a\in\Z^{d-1}$, and use the induction hypothesis applied to each of these functions to obtain
\begin{equation}\label{eq:prophyp}
\|f_b\|_{\textup{U}^k} \leq \|f_b\|_{\ell^p}.
\end{equation}
Note that $f_b$ is identically zero unless $0\leq b\leq n-1$.
By the definition of the Gowers norm we have
\begin{align*} 
\|f\|_{\textup{U}^k}^{2^k}
& = \sum_{\substack{a,h_1,\ldots,h_k\in\Z^{d-1}\\ b,l_1,\ldots,l_k\in\Z}} \prod_{(\epsilon_1,\ldots,\epsilon_k)\in\{0,1\}^k} \mathcal{C}^{\epsilon_1+\cdots+\epsilon_k} f_{b+\epsilon_1 l_1+\cdots+\epsilon_k l_k}(a+\epsilon_1 h_1+\cdots+\epsilon_k h_k) \\
& = \sum_{b,l_1,\ldots,l_k\in\Z} \big\langle (f_{b+\epsilon_1 l_1+\cdots+\epsilon_k l_k})_{(\epsilon_1,\ldots,\epsilon_k)\in\{0,1\}^k} \big\rangle_{\textup{U}^k},
\end{align*}
so the Gowers--Cauchy--Schwarz inequality followed by \eqref{eq:prophyp} yields
\begin{align*} 
\|f\|_{\textup{U}^k}^{2^k}
& \leq \sum_{b,l_1,\ldots,l_k\in\Z} \prod_{(\epsilon_1,\ldots,\epsilon_k)\in\{0,1\}^k} \| f_{b+\epsilon_1 l_1+\cdots+\epsilon_k l_k} \|_{\textup{U}^k} \\
& \leq \sum_{b,l_1,\ldots,l_k\in\Z} \prod_{(\epsilon_1,\ldots,\epsilon_k)\in\{0,1\}^k} \| f_{b+\epsilon_1 l_1+\cdots+\epsilon_k l_k} \|_{\ell^p}.
\end{align*}
It remains to apply \eqref{eq:auxpower3} with
\[ g(b) := \|f_b\|_{\ell^p} \]
to conclude
\[ \|f\|_{\textup{U}^k}^{2^k} \leq \Big( \sum_{b=0}^{n-1} \|f_b\|_{\ell^p}^p \Big)^{2^k/p} = \|f\|_{\ell^p}^{2^k}, \]
which proves \ref{it2}.

\smallskip
\emph{Proof of \ref{it2}$\implies$\ref{it3}}.
This is obvious by taking $f=\mathbbm{1}_A$, observing 
\[ \|\mathbbm{1}_A\|_{\ell^p}^{2^k} = |A|^{2^k/p} = |A|^t, \]
and recalling the definition of $P_k$.

\smallskip
\emph{Proof of \ref{it3}$\implies$\ref{it4}}.
This is obvious by taking $C=1$.

\smallskip
\emph{Proof of \ref{it4}$\implies$\ref{it1}}.
The assumption \ref{it4} can be equivalently stated as: there exist a constant $D\in[1,\infty)$ such that 
\begin{equation}\label{eq:auxpower1}
\|\mathbbm{1}_A\|_{\textup{U}^k} \leq D |A|^{t/2^k} = D \|\mathbbm{1}_A\|_{\ell^p}
\end{equation}
for every $d\geq0$ and every $A\subseteq\{0,1,\ldots,n-1\}^d$.
The following idea is borrowed from \cite[Proof of Lemma 19]{DGIM21}.

We will first prove that for every $d\in\N$ and every function $f\colon\Z^d\to[0,\infty)$ with support in \eqref{eq:thecube} one has
\begin{equation}\label{eq:auxpower2}
\|f\|_{\textup{U}^k} \leq D (3+d\log_2 n) \|f\|_{\ell^p}.
\end{equation}
Denote $M:=\|f\|_{\ell^\infty}$, $N:=\lceil d\log_2 n\rceil$, and write
\[ \frac{f(a)}{M} = \sum_{i=0}^{N} \frac{\beta_i(a)}{2^i} + \frac{f'(a)}{M} \]
for every $a\in\{0,1,\ldots,n-1\}^d$ and some $\beta_i(a)\in\{0,1\}$, $f'(a)\in[0,2^{-N}M)$. 
For each index $i$ define the set
\[ A_i := \{ a \,:\, \beta_i(a)=1 \} \]
and the function
\[ f_i := \frac{M}{2^i} \mathbbm{1}_{A_i} = \frac{M}{2^i} \beta_i. \]
An application of the estimate \eqref{eq:auxpower1} to $A_i$ yields
\[ \|f_i\|_{\textup{U}^k} = \frac{M}{2^i} \|\mathbbm{1}_{A_i}\|_{\textup{U}^k} \leq D \frac{M}{2^i} \|\mathbbm{1}_{A_i}\|_{\ell^p} = D \|f_i\|_{\ell^p}, \]
while $f'$ is trivially controlled as
\[ \|f'\|_{\textup{U}^k} \leq \frac{M}{2^N} \Big( \sum_{a,h_1,\ldots,h_k\in\{0,\ldots,n-1\}^d} 1 \Big)^{1/2^k}
\leq \frac{M}{2^N} (n^d)^{(k+1)/2^k} \leq M (n^d)^{(k+1)/2^k-1} \leq M. \]
Thus, the triangle inequality for the Gowers norm gives
\[ \|f\|_{\textup{U}^k}
= \Big\| \sum_{i=0}^{N} f_i + f' \Big\|_{\textup{U}^k}
\leq \sum_{i=0}^{N} \|f_i\|_{\textup{U}^k} + \|f'\|_{\textup{U}^k}
\leq D \Big( \sum_{i=0}^{N} \|f_i\|_{\ell^p} + \|f\|_{\ell^\infty} \Big). \]
Since $0\leq f_i\leq f$, we clearly have $\|f_i\|_{\ell^p}\leq \|f\|_{\ell^p}$, which concludes
\[ \|f\|_{\textup{U}^k} \leq D (N+2) \|f\|_{\ell^p} \]
and completes the proof of \eqref{eq:auxpower2}.

Finally, we use a tensoring trick to remove the constant. Take an arbitrary $g\colon\Z\to[0,\infty)$ supported in $\{0,1,\ldots,n-1\}$ and, for some $d\in\N$, define $f\colon\Z^d\to[0,\infty)$ to be the $d$-th tensor power of $g$, i.e., 
\[ f(a_1,a_2,\ldots,a_d) := g(a_1) g(a_2) \cdots g(a_d). \]
Then
\[ \|f\|_{\textup{U}^k} = \|g\|_{\textup{U}^k}^d \text{ and } \|f\|_{\ell^p} = \|g\|_{\ell^p}^d, \]
so taking the $d$-th roots of \eqref{eq:auxpower2} gives
\[ \|g\|_{\textup{U}^k} \leq \big( D (3+d\log_2 n) \big)^{1/d} \|g\|_{\ell^p}. \]
Letting $d\to\infty$ we obtain
\[ \|g\|_{\textup{U}^k} \leq \|g\|_{\ell^p} \]
and thus finalize the proof of \ref{it1}.


\section{Proof of Theorem \ref{thm:binary}}
\label{sec:binary}
Fix an integer $k\geq2$ throughout this section and set
\[ t = \log_2 (2k+2), \]
always remembering that $t$ depends on $k$. 
By Proposition \ref{prop:product}, i.e., by formula \eqref{eq:pknvstkn}, we only need to verify
\begin{equation}\label{eq:twopointineq}
\sum_{a,h_1,\ldots,h_k\in\Z} \prod_{(\epsilon_1,\ldots,\epsilon_k)\in\{0,1\}^k} f(a+\epsilon_1 h_1+\cdots+\epsilon_k h_k)
\leq \bigg( f(0)^{2^k/t} + f(1)^{2^k/t} \bigg)^t 
\end{equation}
for every $f\colon\Z\to[0,\infty)$ that vanishes outside $\{0,1\}$.
Substituting 
\[ x=f(0)^{2^k/t},\quad y=f(1)^{2^k/t} \]
we evaluate the left hand side of \eqref{eq:twopointineq} as follows.
\begin{itemize}
\item The term with $a=0$, $h_1=\cdots=h_k=0$ contributes to the sum with $f(0)^{2^k}=x^t$.
\item The term with $a=1$, $h_1=\cdots=h_k=0$ contributes to the sum with $f(1)^{2^k}=y^t$.
\item Terms with $a=0$, $h_{j_0}=1$ for some index $j_0$, $h_{j}=0$ for every index $j\neq j_0$ contribute to the sum with $k f(0)^{2^{k-1}} f(1)^{2^{k-1}} = k x^{t/2} y^{t/2}$.
\item Terms with $a=1$, $h_{j_0}=-1$ for some index $j_0$, $h_{j}=0$ for every index $j\neq j_0$ also contribute to the sum with $k f(0)^{2^{k-1}} f(1)^{2^{k-1}} = k x^{t/2} y^{t/2}$.
\end{itemize}
It remains to prove the generalization of an ``elementary inequality'' by Kane and Tao \cite[Lemma 8]{KT17} stated in Lemma \ref{lm:binary} below. It is similar to but different from \cite[Lemma 9]{DGIM21}, \cite[Lemma 5]{DGIM21}, and \cite[Theorem 1]{Kov23}.

\begin{lemma}\label{lm:binary}
For every $x,y\in[0,\infty)$ we have
\[ x^t + y^t + 2k x^{t/2} y^{t/2} \leq (x+y)^t. \]
\end{lemma}

\begin{proof}
The case $k=2$ is precisely \cite[Lemma 8]{KT17}, which has also been reproved several times in \cite{DGIM21,Kov23}.
Thus, we can assume that $k\geq3$, which also guarantees $t\geq3$.

The inequality is trivial for $x=y=0$. In general, both sides of the inequality are homogeneous of order $t$ in the variables $x$ and $y$, so it is sufficient to prove it when we also normalize $x+y=1$.
Define the function $\varphi\colon[0,1]\to[0,\infty)$ as
\[ \varphi(x) := x^t + (1-x)^t + 2k (x(1-x))^{t/2}. \]
Since $\varphi(1-x)=\varphi(x)$, we only need to show that
\begin{equation}\label{eq:proovingphi}
\varphi(x) \leq 1
\end{equation}
for every $x\in[0,1/2]$.
Let $x_0\in\langle 0,1/2\rangle$ be the unique solution $x=x_0$ to the equation 
\[ k (x(1-x))^{t/2-1} = 1, \]
which exists due to $k(1/4)^{t/2-1}>1$.
We prove \eqref{eq:proovingphi} separately on two sub-intervals of $[0,1/2]$.

\smallskip
\emph{Case 1: interval $[0,x_0]$.}
For every $0\leq x\leq x_0$ we have, by the definition of $x_0$,
\[ 2k (x(1-x))^{t/2} \leq 2x(1-x), \]
which implies
\[ \varphi(x) \leq x^2 + (1-x)^2 + 2x(1-x) = (x+1-x)^2 = 1 \]
and so confirms \eqref{eq:proovingphi}.

\smallskip
\emph{Case 2: interval $\langle x_0,1/2]$.}
Differentiating we get
\[ \varphi'(x) = t x^{t-1} - t (1-x)^{t-1} + k t (x(1-x))^{t/2-1} (1-2x). \]
Also denote $\psi\colon[0,1/2]\to\R$,
\[ \psi(x) := x^{t-1} - (1-x)^{t-1} + 1 - 2x. \]
Note that $\psi$ is concave due to
\[ \psi''(x) = (t-1)(t-2) \big( x^{t-3} - (1-x)^{t-3} \big) \leq 0. \]
Since, $\psi(0)=0=\psi(1/2)$, we also have that $\psi$ is non-negative on $[0,1/2]$.
Now, for every $x\in\langle x_0,1/2]$ we have
\[ k (x(1-x))^{t/2-1} > 1, \]
which implies
\[ \varphi'(x) > t \psi(x) \geq 0. \]
Consequently, $\varphi$ is increasing on the interval $\langle x_0,1/2]$, so
\[ \varphi\Big(\frac{1}{2}\Big) = 2^{1-t} + 2k \cdot 2^{-t} = 1 \]
guarantees that \eqref{eq:proovingphi} holds on that interval too.
\end{proof}


\section{Proof of Theorem \ref{thm:general}}
We separately prove the lower bound
\begin{equation}\label{eq:gen1lower} 
\liminf_{n\to\infty} \bigl((t_{k,n}-k-1) \log_2 n\bigr) 
\geq -\frac{k+1}{2} \log_2 (k+1) + k
\end{equation}
and the upper bound
\begin{equation}\label{eq:gen1upper} 
t_{k,n} \leq k+1 - \frac{c}{\log_2 n},
\end{equation}
where $c>0$ is a constant such that \eqref{eq:Shao} holds.
Before we begin, note that 
\[ \|\mathbbm{1}_{\{0,1,\ldots,n-1\}}\|_{\textup{U}^k} \leq n^{1/p_{k,n}}, \]
so \eqref{eq:comparable} and \eqref{eq:pknvstkn} give trivial bounds
\begin{equation}\label{eq:pkntrivial}
\frac{2^k}{k+1} \leq p_{k,n} \leq \frac{2^k}{k+1} + o_{k}^{n\to\infty}(1)
\end{equation}
and
\begin{equation}\label{eq:tkntrivial}
k+1 - o_{k}^{n\to\infty}(1) \leq t_{k,n} \leq k+1.
\end{equation}


\subsection{The lower bound}
The Gowers norm of a complex function $f$ on the real line is defined as
\[ \|f\|_{\textup{U}^k(\R)} := \biggl( \int_{\R^{k+1}} \prod_{(\epsilon_1,\ldots,\epsilon_k)\in\{0,1\}^k} \mathcal{C}^{\epsilon_1+\cdots+\epsilon_k} f(x+\epsilon_1 h_1+\cdots+\epsilon_k h_k) \,\textup{d}x \,\textup{d}h_1 \cdots \textup{d}h_k \biggr)^{1/2^k}. \]
By a result of Eisner and Tao \cite[Theorem 1.12]{ET12} the optimal constant $C_k$ in the inequality
\[ \|g\|_{\textup{U}^k(\R)} \leq C_k \|g\|_{\textup{L}^{2^k/(k+1)}(\R)} \]
on the real line equals
\begin{equation}\label{eq:ETconstant}
C_k = \frac{2^{k/2^k}}{(k+1)^{(k+1)/2^{k+1}}}.
\end{equation}
The equality is attained for all Gaussian functions among other extremizers, namely, for $k\geq2$, the Gaussians modulated by complex trigonometric polynomials of degree at most $k-1$.
This gives us an idea to test inequality \eqref{eq:GowersLp} against truncated discrete analogues of a Gaussian. This idea was already employed by Shao \cite{Sha24}, who did not work with the Gowers norms, but rather motivated the choice by the fact that the Gaussians also extremize real line Young's convolution inequality.

For a parameter $M>1$ and an integer $n\geq2$ define a function $f_{M,n}\colon\Z\to[0,\infty)$ by
\[ f_{M,n}(m) := \begin{cases}
\exp\biggl( -4M^2 \Big( \displaystyle\frac{m}{n} - \frac{1}{2} \Big)^2 \biggr) & \text{for } m\in\{0,1,2,\ldots,n-1\}, \\
0 & \text{otherwise}. 
\end{cases} \]
We also introduce $g,g_M\colon\R\to[0,\infty)$ by the formulae
\begin{align*}
g(x) & := e^{-x^2}, \\
g_M(x) & := \exp\biggl( -4M^2 \Big(x-\frac{1}{2}\Big)^2\biggr) = g(2Mx-M).
\end{align*}
Writing the $(k+1)$-dimensional integral as a limit of its Riemann sums we obtain
\begin{equation}\label{eq:auxlim1}
\lim_{n\to\infty} \frac{1}{n^{k+1}} \|f_{M,n}\|_{\textup{U}^k}^{2^k}
= \|g_M\mathbbm{1}_{[0,1]}\|_{\textup{U}^k(\R)}^{2^k} 
= \frac{1}{(2M)^{k+1}} \|g\mathbbm{1}_{[-M,M]}\|_{\textup{U}^k(\R)}^{2^k}. 
\end{equation}
Also, \eqref{eq:pkntrivial} implies
\begin{align*} 
\|f_{M,n}\|_{\ell^{p_{k,n}}}^{p_{k,n}} 
& = \sum_{m=0}^{n-1} \exp\biggl( -4M^2 p_{k,n} \Big( \frac{m}{n} - \frac{1}{2} \Big)^2 \biggr) \\
& \leq \sum_{m=0}^{n-1} \exp\biggl( -4M^2 \frac{2^k}{k+1} \Big( \frac{m}{n} - \frac{1}{2} \Big)^2 \biggr)
= \|f_{M,n}\|_{\ell^{2^k/(k+1)}}^{2^k/(k+1)} ,
\end{align*}
so we similarly obtain
\begin{align*}
& \limsup_{n\to\infty} \frac{1}{n} \|f_{M,n}\|_{\ell^{p_{k,n}}}^{p_{k,n}}
\leq \lim_{n\to\infty} \frac{1}{n} \|f_{M,n}\|_{\ell^{2^k/(k+1)}}^{2^k/(k+1)} \\
& = \|g_M \mathbbm{1}_{[0,1]}\|_{\textup{L}^{2^k/(k+1)}(\R)}^{2^k/(k+1)}
= \frac{1}{2M} \|g\mathbbm{1}_{[-M,M]}\|_{\textup{L}^{2^k/(k+1)}(\R)}^{2^k/(k+1)}.
\end{align*}
Moreover, by \eqref{eq:pknvstkn} and \eqref{eq:tkntrivial},
\begin{align}
& \limsup_{n\to\infty} \frac{1}{n^{t_{k,n}}} \|f_{M,n}\|_{\ell^{p_{k,n}}}^{2^k} \nonumber \\
& = \limsup_{n\to\infty} \Bigl( \frac{1}{n} \|f_{M,n}\|_{\ell^{p_{k,n}}}^{p_{k,n}} \Bigr)^{t_{k,n}} 
= \Bigl( \limsup_{n\to\infty} \frac{1}{n} \|f_{M,n}\|_{\ell^{p_{k,n}}}^{p_{k,n}} \Bigr)^{\displaystyle\lim_{n\to\infty} t_{k,n}} \nonumber \\
& = \Bigl( \limsup_{n\to\infty} \frac{1}{n} \|f_{M,n}\|_{\ell^{p_{k,n}}}^{p_{k,n}} \Bigr)^{k+1}
\leq \frac{1}{(2M)^{k+1}} \|g\mathbbm{1}_{[-M,M]}\|_{\textup{L}^{2^k/(k+1)}(\R)}^{2^k}. \label{eq:auxlim2}
\end{align}
Dividing \eqref{eq:auxlim1} by \eqref{eq:auxlim2} we conclude
\[ \liminf_{n\to\infty} n^{t_{k,n}-k-1} \Big(\frac{\|f_{M,n}\|_{\textup{U}^k}}{\|f_{M,n}\|_{\ell^{p_{k,n}}}}\Big)^{2^k}
\geq \Big(\frac{\|g\mathbbm{1}_{[-M,M]}\|_{\textup{U}^k(\R)}}{\|g\mathbbm{1}_{[-M,M]}\|_{\textup{L}^{2^k/(k+1)}(\R)}}\Big)^{2^k}. \]
The definition of $p_{k,n}$ guarantees
\[ \|f_{M,n}\|_{\textup{U}^k} \leq \|f_{M,n}\|_{\ell^{p_{k,n}}}, \]
so, taking logarithms, we obtain
\[ \liminf_{n\to\infty} \bigl((t_{k,n}-k-1) \log_2 n\bigr) 
\geq 2^k \log_2 \frac{\|g\mathbbm{1}_{[-M,M]}\|_{\textup{U}^k(\R)}}{\|g\mathbbm{1}_{[-M,M]}\|_{\textup{L}^{2^k/(k+1)}(\R)}}. \]
Finally, we can take the limit as $M\to\infty$ of the right hand side and obtain
\[ \liminf_{n\to\infty} \bigl((t_{k,n}-k-1) \log_2 n\bigr) 
\geq 2^k \log_2 \frac{\|g\|_{\textup{U}^k(\R)}}{\|g\|_{\textup{L}^{2^k/(k+1)}(\R)}} = 2^k \log_2 C_k, \]
which, by \eqref{eq:ETconstant}, is precisely \eqref{eq:gen1lower}.


\subsection{The upper bound}
We will show \eqref{eq:gen1upper} by induction on $k\geq2$. The base $k=2$ is precisely \eqref{eq:Shao}.
The induction step follows from $t_{k+1,n}\leq t_{k,n}+1$, which is a clear consequence of the following implication: if
\begin{equation}\label{eq:upperind}
\|f\|_{\textup{U}^k}^{2^k} \leq \big\| |f|^{2^k/t} \big\|_{\ell^1}^{t},
\end{equation}
holds for every function $f$ supported on \eqref{eq:thecube}, then one also has
\begin{equation}\label{eq:upperind2}
\|f\|_{\textup{U}^{k+1}}^{2^{k+1}} \leq \big\| |f|^{2^{k+1}/(t+1)} \big\|_{\ell^1}^{t+1}
\end{equation}
for every function $f$ supported on the same cube.
To verify the implication, we denote $g=|f|^{2^k/t}$. Applying \eqref{eq:upperind} to the function $x\mapsto\overline{f(x+h)}f(x)$ for each fixed $h$ and using Young's convolution inequality \eqref{eq:Youngineq} gives
\begin{align*} 
\|f\|_{\textup{U}^{k+1}}^{2^{k+1}} 
& = \sum_{h\in\Z^d} \big\|\overline{f(\cdot+h)}f(\cdot)\big\|_{\textup{U}^k}^{2^k}
\leq \sum_{h\in\Z^d} \big\| |f(\cdot+h)|^{2^k/t} |f(\cdot)|^{2^k/t} \big\|_{\ell^{1}}^{t} \\
& = \|g\ast\widetilde{g}\|_{\ell^t}^{t} 
\leq \|g\|_{\ell^{2t/(t+1)}}^{t} \|\widetilde{g}\|_{\ell^{2t/(t+1)}}^{t}
= \|g\|_{\ell^{2t/(t+1)}}^{2t} = \big\| |f|^{2^{k+1}/(t+1)} \big\|_{\ell^1}^{t+1}. 
\end{align*}
Thus, \eqref{eq:upperind} implies \eqref{eq:upperind2}.


\section{Proof of Theorem \ref{thm:general2}}
\label{sec:proofofkthm}
Before we begin with the proof of Theorem \ref{thm:general2}, we need to formulate and prove several auxiliary inequalities for the Shannon entropy of special distributions, some of which we could not find in the literature.
Afterwards, we will separately prove the lower bound
\begin{equation}\label{eq:gen2lower} 
\liminf_{k\to\infty} \Bigl( t_{k,n} - \frac{(n-1)\log_2 (2k) - \log_2 (n-1)!}{H_{n-1}} \Bigr) \geq 0
\end{equation}
and the upper bound
\begin{equation}\label{eq:gen2upper} 
\limsup_{k\to\infty} \Bigl( t_{k,n} - \frac{(n-1)\log_2 (2k) - \log_2 (n-1)!}{H_{n-1}} \Bigr) \leq 0 .
\end{equation}

\subsection{Entropy estimates}
\label{subsec:entropy}
Let $X_1,X_2,X_3,\ldots$ be independent symmetric Bernoulli random variables, namely
\[ X_i \sim \begin{pmatrix}
0 & 1 \\
1/2 & 1/2
\end{pmatrix}, \]
so that $X_1+X_2+\cdots+X_m\sim B(m,1/2)$.
This subsection will be concerned with random variabes of the form
\[ h_1 X_1 + \cdots + h_m X_m \]
for arbitrary non-zero integers $h_1,\ldots,h_m$.

First, we need precise estimates for the entropy
\[ H_m = \textup{H}(X_1 + \cdots + X_m) \] 
of the symmetric binomial distribution $B(m,1/2)$, which is the same number that was introduced in \eqref{eq:binomentr}.
The asymptotic series for the numbers $H_m$ as $m\to\infty$ has been derived by several authors \cite{JS99,K98,F99}, but we need exact inequalities. They are available from the work of Adell, Lekuona, and Yu \cite{ALY10}, who strengthened the previous bounds by Chang and Weldon \cite{CW79}.

\begin{lemma}[\cite{ALY10}]\label{lm:entropy2}
We have
\begin{equation}\label{eq:sharpHmest}
\frac{1}{2} \log_2 \frac{e \pi m}{2} - \frac{1}{4m} < H_m < \frac{1}{2} \log_2 \frac{e \pi m}{2} + \frac{1}{10m}
\end{equation}
for every $m\in\N$.
\end{lemma}

\begin{proof}
The authors of \cite{ALY10} used the natural logarithm in the definition of the entropy (i.e., they measure it in nats), whereas we chose the base-$2$-logarithm in \eqref{eq:entrdef} (i.e., the entropy is measured in bits).
Apart from that modification, \cite[Formula (7)]{ALY10}, which is a consequence of \cite[Corollary 1]{ALY10}, specialized to the symmetric binomial distribution reads
\[ \frac{C_1}{m} + \frac{C_2}{m^2} + \frac{C_3}{m^3} < H_m - \frac{1}{2} \log_2 \frac{e \pi m}{2} < \frac{C_4}{m}, \]
where
\begin{align*}
C_1 & = -0.24606\ldots, \\
C_2 & = 0.17527\ldots, \\
C_3 & = -0.00400\ldots, \\
C_4 & = 0.08202\ldots.
\end{align*}
It remains to disregard the terms $C_2/m^2$ and $C_3/m^3$ on the left hand side, since their sum is always positive.
\end{proof}

\begin{corollary}\label{cor:decreasing}
The sequence $(H_m/m)_{m=1}^{\infty}$ is strictly decreasing, i.e., 
\[ \frac{H_1}{1} > \frac{H_2}{2} > \frac{H_3}{3} > \frac{H_4}{4} > \cdots. \]
\end{corollary}

\begin{proof}
For any integer $m\geq2$ inequality
\[ (m+1) H_m > m H_{m+1} \]
will follow from Lemma \ref{lm:entropy2} and
\[ (m+1) \Big( \frac{1}{2} \log_2 \frac{e \pi m}{2} - \frac{1}{4m} \Big) 
> m \Big( \frac{1}{2} \log_2 \frac{e \pi (m+1)}{2} + \frac{1}{10(m+1)} \Big). \]
However, the last estimate can be multiplied by $2$ and equivalently rewritten as
\[ \underbrace{\log_2 m}_{\geq 1} 
+ \underbrace{\log_2\frac{\pi}{2}}_{>0}
+ \underbrace{\log_2 e - \log_2 \Big(1+\frac{1}{m}\Big)^m}_{>0} 
> \underbrace{\frac{m}{5(m+1)} + \frac{m+1}{2m}}_{<1/5+3/4<1}, \]
which clearly holds.
Finally, the remaining inequality between the first two binomial entropies is trivial because of $H_1=1$, $H_2=3/2$.
\end{proof}

We turn to formulating and proving more general inequalities for the Shannon entropy that will play a crucial role in the proof of Theorem \ref{thm:general2}.
Recall that $p_Z$ denotes the probability mass function of a discrete, integer-valued random variable $Z$; see Section \ref{sec:notation}.
Also recall that $p_Z^{\downarrow}$ is its decreasing rearrangement.

Let us begin with a purely combinatorial result.

\begin{lemma}\label{lm:majorization}
Let $h_1,\ldots,h_m$ be nonzero integers, $X = h_1X_1+\ldots+h_mX_m$, and $Y \sim B(m,1/2)$. Then $p_Y^{\downarrow}$ majorizes $p_X^{\downarrow}$. Moreover, we have $p_X^{\downarrow} = p_{Y}^{\downarrow}$ if and only if $|h_1| = \ldots = |h_m|$.
\end{lemma}

Here we view $p_X^{\downarrow}$ and $p_Y^{\downarrow}$ as $N$-tuples for a positive integer $N$ sufficiently large that both of their supports are contained in $\{1,2,\ldots,N\}$. In this situation the notion of ``majorization'' from Section \ref{sec:notation} applies too.

\begin{proof}
Changing some signs among $h_1,\ldots,h_m$ only translates the distribution of $h_1 X_1 + \cdots + h_m X_m$ along the integers and preserves the entropy, because
\[ - h_i X_i = h_i (1-X_i) - h_i \sim h_i X_i - h_i. \]
Thus, we can assume that all numbers $h_i$ are positive integers. 
Also, the case $m=1$ is trivial because of $X\sim Y$, so we assume that $m\geq2$.

Consider, for each $z \in \Z$, the family
\[ \mathcal{A}_z := \Bigl\{A \subseteq \{1,2,\ldots,m\} \,:\, \sum_{i\in A}h_i = z\Bigr\}, \]
which is readily seen to be an antichain. 
It is clear that
\[ \sum_{j=1}^{\infty} p_Y^{\downarrow}(j) = 1 = \sum_{j=1}^{\infty} p_X^{\downarrow}(j). \]
Furthermore, given any $n \in \mathbb{N}$, we want to show that
\begin{equation}\label{eq:mass_functions}
\sum_{j=1}^{n} p_Y^{\downarrow}(j) \geq \sum_{j=1}^{n}p_X^{\downarrow}(j).
\end{equation}
On the one hand, observe that 
\[ p_Y^{\downarrow}(1),\ p_Y^{\downarrow}(2),\ p_Y^{\downarrow}(3),\ p_Y^{\downarrow}(4),\ \ldots \]
appearing on the left hand side of \eqref{eq:mass_functions} are simply ($1/2^{m}$)-multiples of the binomial coefficients with $m$ on top, sorted in the descending order; namely,
\[ 2^{-m}\binom{m}{\lfloor m/2\rfloor},\, 2^{-m}\binom{m}{\lfloor m/2\rfloor+1},\, 2^{-m}\binom{m}{\lfloor m/2\rfloor-1},\, 2^{-m}\binom{m}{\lfloor m/2\rfloor+2},\, \ldots. \]
On the other hand, choose a bijection $\pi \colon \mathbb{N} \to \mathbb{Z}$ such that $p_X^{\downarrow} = p_X \circ \pi$.
The right-hand side of \eqref{eq:mass_functions} equals $|\mathcal{B}_n|/2^m$, where we define
\[ \mathcal{B}_n := \bigcup_{j=1}^{n}\mathcal{A}_{\pi(j)}. \]
Note that $\mathcal{B}_n$ is a union of $n$ antichains in $\mathcal{P}(\{1,2,\ldots,m\})$, so, in particular, any chain in $\mathcal{B}_n$ has size at most $n$. 
Erd\H{o}s \cite{Erd45} proved a generalization of Sperner's theorem bounding the maximal size of a subcollection of $\mathcal{P}(\{1,2,\ldots,m\})$ that does not have a chain of length longer than $n$.
Namely, by \cite[Theorem 5]{Erd45}, it follows that $|\mathcal{B}_n|$ is at most the sum of the $n$ largest binomial coefficients of the form $\binom{m}{k}$ with $0 \leq k \leq m$. Therefore, the inequality \eqref{eq:mass_functions} follows.

Turning to the second statement, it is clear that $h_1 = \ldots = h_m$ implies $p_X^{\downarrow} = p_Y^{\downarrow}$. To prove the converse, note that $p_X^{\downarrow}(1) = p_Y^{\downarrow}(1)$ implies that $|\mathcal{A}_{\pi(1)}| = \binom{m}{\lfloor m/2\rfloor}$. Since $\mathcal{A}_{\pi(1)}$ is an antichain, it follows from the equality case in Sperner's theorem (see e.g.\ \cite[Theorem 1.2.2]{anderson}) that $\mathcal{A}_{\pi(1)}$ is the family of all subsets of $\{1,2,\ldots,m\}$ of some fixed size $k \in \{\lfloor m/2 \rfloor, \lceil m/2\rceil\}$. But now, for any distinct $i,j\in\{1,\ldots,m\}$, we may choose a subset $A \subseteq \{1,\ldots,m\}\setminus\{i,j\}$ of size $k-1$. In particular, $A \cup \{i\}$ and $A\cup\{j\}$ both belong to $\mathcal{A}_{\pi(1)}$. This means that
\[ h_i + \sum_{a\in A}h_a = \sum_{a\in A\cup\{i\}}h_a = \pi(1) = \sum_{a\in A\cup\{j\}}h_a = h_j + \sum_{a\in A}h_a, \]
whence $h_i = h_j$, as desired.
\end{proof}

\begin{remark}
The fact that the size of a union of $n$ antichains in $\mathcal{P}(\{1,2,\ldots,m\})$ is maximized by taking the $n$ largest layers can also be easily proved by means of the Yamamoto--Bollob\'{a}s--Lubell--Meshalkin inequality \cite{lubell}. 
\end{remark}

The following consequence of the previous lemma can be thought of as an entropic variant of the Littlewood--Offord problem \cite{Erd45}.

\begin{proposition}\label{prop:niceentropy}
For every $m\in\N$ and arbitrary $h_1,\ldots,h_m\in\Z\setminus\{0\}$ we have
\[ \textup{H}(h_1 X_1 + \cdots + h_m X_m) \geq H_m, \]
with equality attained only when $|h_1|=\cdots=|h_m|$.
\end{proposition}

Note that the case of equality, namely $|h_1|=\cdots=|h_m|$, means precisely that $h_1 X_1 + \cdots + h_m X_m$ is distributed as a nondegenerate affine transformation of the symmetric binomial distribution $B(m,1/2)$.

\begin{proof}
Consider the function
\[ \psi \colon [0,1] \to \mathbb{R}, \quad \psi(x) := \begin{cases} -x\log_2 x & \text{for } x>0, \\ 0 & \text{for } x=0. \end{cases} \]
Then the entropy \eqref{eq:entrdef} of an arbitrary integer-valued random variable $Z$ taking only finitely many values can be expressed as
\[ H(Z) = \sum_{z\in\Z} \psi(p_Z(z)) = \sum_{n=1}^{\infty} \psi(p_Z^{\downarrow}(n)). \]
Again take $X = h_1X_1+\ldots+h_mX_m$ and $Y \sim B(m,1/2)$.
From Lemma \ref{lm:majorization} we know that $p_Y^{\downarrow}$ majorizes $p_X^{\downarrow}$.
Since $\psi$ is strictly concave, the desired estimate,
\[ H(X) \geq H(Y), \]
follows from Karamata's inequality recalled in Section \ref{sec:notation}.
Equality holds if and only if $p_X^{\downarrow}=p_Y^{\downarrow}$, which is, by Lemma \ref{lm:majorization} again, satisfied only when all $h_i$ are the same in the absolute value.
\end{proof}

\begin{corollary}\label{cor:entropy}
For every $n\geq2$, $1\leq l\leq n-1$, and $h_1,\ldots,h_l\in \Z\setminus\{0\}$ such that
\begin{equation}\label{eq:corsum}
|h_1|+\cdots+|h_l|\leq n-1
\end{equation} 
we have
\begin{equation}\label{eq:mainetropyin}
\frac{\textup{H}(h_1 X_1 + \cdots + h_l X_l)}{l} \geq \frac{H_{n-1}}{n-1},
\end{equation}
with equality attained only when $l=n-1$ and $|h_1|=\cdots=|h_{n-1}|=1$, in which case $h_1 X_1 + \cdots + h_{n-1} X_{n-1}$ is distributed as a translate of $B(n-1,1/2)$.
\end{corollary}

\begin{proof}
If the number of terms $l$ is fixed, then, by Proposition \ref{prop:niceentropy}, the left hand side of \eqref{eq:mainetropyin} is at least $H_l/l$ and then the strict inequality for $1\leq l\leq n-2$ follows from Corollary \ref{cor:decreasing}. 
The inequality also holds when $l=n-1$ and, by Proposition \ref{prop:niceentropy} again, we must additionally have $|h_1|=\cdots=|h_{n-1}|$ in the case of an equality. This combined with the assumption \eqref{eq:corsum} means that all these absolute values need to be $1$.
\end{proof}

\begin{remark}
Corollary \ref{cor:entropy} will be sufficient for our intended application and it allows for a conceptually simpler, but more computational proof.
Namely, the most classical result on the Littlewood--Offord problem \cite[Theorem 1]{Erd45},
\[ \max_{z\in\Z} \mathbb{P}(h_1 X_1 + \cdots + h_l X_l=z) \leq \binom{l}{\lfloor l/2\rfloor} 2^{-l}, \]
gives a ``cheaper'' lower bound
\[ \textup{H}(h_1 X_1 + \cdots + h_l X_l) > \frac{1}{2} \log_2 \frac{\pi l}{2}. \]
For $n>100$ and $1\leq l\leq 3(n-1)/4$ we then use Lemma \ref{lm:entropy2} and some elementary calculus to show
\begin{align*}
& 2(n-1)\textup{H}(h_1 X_1 + \cdots + h_l X_l) - 2l H_{n-1} \\
& > (n-1) \log_2 \frac{\pi l}{2} - 2l H_{n-1} \\
& > (n-1) \log_2 \frac{\pi l}{2} - l \log_2 \frac{e \pi (n-1)}{2} - \frac{l}{5(n-1)} >0. 
\end{align*}
On the other hand, when $n>100$ and $3(n-1)/4<l\leq n-2$, the assumption \eqref{eq:corsum} forces at least $2l-n+1$ of the numbers $h_i$ to be $\pm1$, which easily gives
\[ \textup{H}(h_1 X_1 + \cdots + h_l X_l) \geq H_{2l-n+1}. \]
Then Lemma \ref{lm:entropy2} and basic calculus yield
\begin{align*}
& 2(n-1)\textup{H}(h_1 X_1 + \cdots + h_l X_l) - 2l H_{n-1} \\
& \geq 2(n-1) H_{2l-n+1} - 2l H_{n-1} \\
& > (n-1) \log_2 \frac{e\pi(2l-n+1)}{2} - \frac{n-1}{2(2l-n+1)} - l \log_2 \frac{e\pi(n-1)}{2} - \frac{l}{5(n-1)} > 0.
\end{align*}
Finally, if $n\leq 100$, then \eqref{eq:corsum} leaves only finitely many cases to check, and a combination of the previous reasoning and a computer-assisted verification handles all of them.
\end{remark}


\subsection{The lower bound}
Recall that, by Proposition \ref{prop:product}, the number $t_{k,n}$ is the smallest $t>0$ such that 
\begin{equation}\label{eq:gen2temp}
\sum_{a,h_1,\ldots,h_k\in\Z} \prod_{(\epsilon_1,\ldots,\epsilon_k)\in\{0,1\}^k} f(a+\epsilon_1 h_1+\cdots+\epsilon_k h_k)
\leq \bigg( \sum_{j=0}^{n-1} f(j)^{2^k/t} \bigg)^t 
\end{equation}
holds for every function $f\colon\Z\to[0,\infty)$ supported in $\{0,1,\ldots,n-1\}$.
On the left hand side of inequality \eqref{eq:gen2temp} we only observe the mutually equal terms obtained by taking:
\begin{itemize}
\item $a\in\{0,1,\ldots,n-1\}$ arbitrary, 
\item precisely $a$ of the numbers $h_1,\ldots,h_k$ to be equal to $-1$,
\item precisely $n-1-a$ of the numbers $h_1,\ldots,h_k$ to be equal to $1$, and
\item precisely $k-n+1$ of the numbers $h_1,\ldots,h_k$ to be equal to $0$.
\end{itemize}
There are 
\[ \binom{k}{n-1} \sum_{a=0}^{n-1} \binom{n-1}{a} = \binom{k}{n-1} 2^{n-1} \]
such choices altogether, and they all contribute to the sum in \eqref{eq:gen2temp} with the same term
\[ \prod_{j=0}^{n-1} f(j)^{\binom{n-1}{j}2^{k-n+1}}. \]
Thus, by using inequality \eqref{eq:gen2temp} with $t=t_{k,n}$, we obtain
\[ \binom{k}{n-1} 2^{n-1} \prod_{j=0}^{n-1} f(j)^{\binom{n-1}{j}2^{k-n+1}} \leq \bigg( \sum_{j=0}^{n-1} f(j)^{2^k/t_{k,n}} \bigg)^{t_{k,n}}. \]
Now take a particular function $f$ defined as
\[ f(j) := \biggl( \frac{\binom{n-1}{j}}{2^{n-1}} \biggr)^{t_{k,n}/2^k} \]
for $0\leq j\leq n-1$ and $f(j):=0$ otherwise, so that the last inequality gives us
\[ \binom{k}{n-1} 2^{n-1} \prod_{j=0}^{n-1} \biggl( \frac{\binom{n-1}{j}}{2^{n-1}} \biggr)^{\binom{n-1}{j}t_{k,n}/2^{n-1}}
\leq \bigg( \sum_{j=0}^{n-1} \frac{\binom{n-1}{j}}{2^{n-1}} \bigg)^{t_{k,n}} =  1. \]
Taking logarithms,
\[ \sum_{j=0}^{n-2}\log_2 (k-j) - \log_2 (n-1)! + (n-1) + \underbrace{\sum_{j=0}^{n-1} \frac{\binom{n-1}{j}t_{k,n}}{2^{n-1}} \log_2 \frac{\binom{n-1}{j}}{2^{n-1}}}_{- H_{n-1} t_{k,n}} \leq 0, \]
i.e.,
\[ \liminf_{k\to\infty} \bigl( H_{n-1} t_{k,n} - (n-1) \log_2 (2k) + \log_2 (n-1)! \bigr)
\geq \lim_{k\to\infty} \sum_{j=0}^{n-2}\log_2 (1-j/k)
= 0, \]
which proves \eqref{eq:gen2lower}.


\subsection{The upper bound}
Note that the term on the left hand side of \eqref{eq:gen2temp} corresponding to at least $n$ nonzero numbers $h_1,\ldots,h_k$ must be identically $0$. 
At the other extreme are the terms when $h_1=\cdots=h_k=0$.
All remaining terms can be grouped by choosing $1\leq l\leq n-1$ nonzero numbers among $h_1,\ldots,h_k$ in $\binom{k}{l}$ ways.
Therefore, for a fixed integer $n\geq 2$ and every $l\in\{1,\ldots,n-1\}$ we define
\begin{align*} 
T_{n,l} := \bigl\{ (a,h_1,\ldots,h_l) \in \Z \times (\Z\setminus\{0\})^l \,:\ & 0\leq a + \epsilon_1 h_1 + \cdots + \epsilon_l h_l\leq n-1 \\
& \text{for every } (\epsilon_1,\ldots,\epsilon_l)\in\{0,1\}^l \bigr\},
\end{align*}
so that inequality \eqref{eq:gen2temp} can now be rewritten as
\begin{align*} 
\sum_{j=0}^{n-1} f(j)^{2^k} + \sum_{l=1}^{n-1} \sum_{(a,h_1,\ldots,h_l)\in T_{n,l}} \binom{k}{l} \prod_{(\epsilon_1,\ldots,\epsilon_l)\in\{0,1\}^l} f(a+\epsilon_1 h_1+\cdots+\epsilon_l h_l)^{2^{k-l}} & \\[-2mm]
\leq \bigg( \sum_{j=0}^{n-1} f(j)^{2^k/t} \bigg)^t & .
\end{align*}
By substituting $g(j)=f(j)^{2^k/t}$, the estimate \eqref{eq:gen2temp} is further equivalent to
\begin{equation}\label{eq:gen2ineq}
\sum_{j=0}^{n-1} g(j)^{t} + \sum_{l=1}^{n-1} \sum_{(a,h_1,\ldots,h_l)\in T_{n,l}} \binom{k}{l} \prod_{(\epsilon_1,\ldots,\epsilon_l)\in\{0,1\}^l} g(a+\epsilon_1 h_1+\cdots+\epsilon_l h_l)^{t/2^l} \leq 1
\end{equation}
for every $g\colon\{0,1,\ldots,n-1\}\to[0,\infty)$ such that $\sum_{j=0}^{n-1}g(j)=1$.
It is understood that only such functions $g$ are considered in the rest of this subsection.
We will prove \eqref{eq:gen2ineq} for any given $0<\delta<1$, for
\begin{equation}\label{eq:choiceoft}
t = \frac{(n-1)\log_2 (2k) - \log_2 (n-1)!}{H_{n-1}} + \delta,
\end{equation}
and for all positive integers $k$ that are sufficiently large depending on $n$ and $\delta$. This will establish \eqref{eq:gen2upper} by showing that the upper limit in question is at most $\delta$ for every $\delta>0$.

Let us first make some preliminary observations.
Every summand, other than $g(j)^t$, on the left hand side of \eqref{eq:gen2ineq} is of the form
\begin{equation}\label{eq:gen2summand}
\binom{k}{l} \bigl( g(0)^{q_0} g(1)^{q_1} \cdots g(n-1)^{q_{n-1}} \bigr)^t
\end{equation}
where $(q_0,q_1,\ldots,q_{n-1})$ is some discrete probability distribution on the set $\{0,1,\ldots,n-1\}$,
which is also the distribution of some random variable $a + h_1 X_1 + \cdots + h_l X_l$ for $(a,h_1,\ldots,h_l)\in T_{n,l}$.
We always interpret $0^0=1$.
For every $g$ as before the inequality between the weighted arithmetic and geometric means gives
\[ 1 = \sum_{j=0}^{n-1} g(j) \geq \sum_{\substack{0\leq j\leq n-1\\ q_j\neq0}} q_j \frac{g(j)}{q_j} \geq \prod_{\substack{0\leq j\leq n-1\\ q_j\neq0}} \Big(\frac{g(j)}{q_j}\Big)^{q_j}
= \frac{\prod_{j=0}^{n-1} g(j)^{q_j}}{\prod_{j=0}^{n-1} q_j^{q_j}}. \]
That way we have obtained
\begin{equation}\label{eq:justag}
\prod_{j=0}^{n-1} g(j)^{q_j} \leq 2^{-\textup{H}(q_0,\ldots,q_{n-1})},
\end{equation}
where we recall that $\textup{H}(q_0,\ldots,q_{n-1})$ is defined by the formula \eqref{eq:entrdef}.
Denote
\[ \theta = 2^{-n 2^n H_{n-1}/(n-1)}. \]
The proof of \eqref{eq:gen2ineq} is split into two cases.
Once again, we always assume that $g$ attains non-negative values that sum to $1$.

\smallskip
\emph{Case 1: for every $0\leq j\leq n-1$ we have $g(j)\leq 1-\theta$.}
Each summand \eqref{eq:gen2summand} from the left hand side of \eqref{eq:gen2ineq} corresponding to some $(a,h_1,\ldots,h_l)\in T_{n,l}$ is, by inequality \eqref{eq:justag}, at most
\[ \binom{k}{l} 2^{-\textup{H}(h_1 X_1 + \cdots + h_l X_l) t}, \]
which is, by the choice \eqref{eq:choiceoft} for $t$, less than or equal to
\begin{numcases}{}
2^{-n+1} \cdot 2^{-H_{n-1}\delta} & \text{for } l=n-1, \label{eq:c1} \\
O_{n}^{k\to\infty}\bigl(k^{l-(n-1)\textup{H}(h_1 X_1 + \cdots + h_l X_l)/H_{n-1}}\bigr) & \text{for } 1\leq l\leq n-2. \label{eq:c2} 
\end{numcases} 
In \eqref{eq:c1} we used the fact that the only possibility to have $l=n-1$ is $a + h_1 X_1 + \cdots + h_{n-1} X_{n-1} \sim B(n-1,1/2)$.
Since $T_{n,n-1}$ has precisely $2^{n-1}$ elements described in the previous subsection and each corresponding term is bounded by \eqref{eq:c1}, we conclude
\[ \sum_{(a,h_1,\ldots,h_{n-1})\in T_{n,n-1}} \binom{k}{n-1} \prod_{(\epsilon_1,\ldots,\epsilon_{n-1})\in\{0,1\}^{n-1}}\!\!\! g(a+\epsilon_1 h_1+\cdots+\epsilon_{n-1} h_{n-1})^{t/2^{n-1}} \leq 2^{-H_{n-1}\delta}. \]
Next, each term corresponding to $T_{n,l}$, $1\leq l\leq n-2$, is bounded by \eqref{eq:c2} and the exponent of $k$ is strictly negative by Corollary \ref{cor:entropy}. Thus,
\[ \sum_{l=1}^{n-2} \sum_{(a,h_1,\ldots,h_l)\in T_{n,l}} \binom{k}{l} \prod_{(\epsilon_1,\ldots,\epsilon_l)\in\{0,1\}^l} g(a+\epsilon_1 h_1+\cdots+\epsilon_l h_l)^{t/2^l} = o_{n}^{k\to\infty}(1). \]
Finally, under the standing assumptions on $g$ and by formula \eqref{eq:choiceoft},
\[ \sum_{j=0}^{n-1} g(j)^{t} \leq n (1-\theta)^t = o_{n}^{k\to\infty}(1). \]
Altogether, the left hand side of \eqref{eq:gen2ineq} is bounded by
\[ \underbrace{2^{-H_{n-1}\delta}}_{<1} + o_{n}^{k\to\infty}(1) \]
for all functions $g$ considered in this case. This is certainly less than $1$ for sufficiently large $k$, just as desired.

\smallskip
\emph{Case 2: for some $0\leq j_0\leq n-1$ we have $g(j_0)>1-\theta$.}
As a consequence of $\sum_j g(j) = 1$ we also have $g(j)<\theta$ for each $j\neq j_0$. 
Denote
\[ N := \sum_{l=1}^{n-1} |T_{n,l}|. \]
Take $k$ sufficiently large so that $t>2^n$; a further largeness requirement on $k$ will be imposed later.
The terms $g(j)^t$ in \eqref{eq:gen2ineq} are here simply controlled as
\[ g(j)^t \leq g(j)^2 \]
for $j=0,1,\ldots,n-1$.
Each of the remaining $N$ terms in \eqref{eq:gen2ineq} is at most
\begin{equation}\label{eq:lastcase}
k^{n-1} g(j_1)^{t/2^n} g(j_2)^{t/2^n}
\end{equation}
for some distinct indices $j_1$ and $j_2$.
Since at least one of these indices is different from $j_0$, we can bound \eqref{eq:lastcase} from the above by
\begin{align*} 
g(j_1) g(j_2) k^{n-1} \theta^{t/2^n-1}
& = g(j_1) g(j_2) k^{n-1} O_{n}^{k\to\infty}(k^{-n}) \\
& = g(j_1) g(j_2) O_{n}^{k\to\infty}(k^{-1})
\leq \frac{2}{N} g(j_1) g(j_2), 
\end{align*}
where the last inequality holds as soon as $k$ is large enough depending on $n$.
Altogether, the left hand side of \eqref{eq:gen2ineq} is then less than or equal to
\[ \sum_{j=0}^{n-1} g(j)^2 + N \cdot \frac{2}{N} \max_{j_1\neq j_2} g(j_1) g(j_2) 
\leq \Big( \sum_{j=0}^{n-1} g(j) \Big)^2 = 1. \]

\smallskip
Both cases are now complete and they finalize the proof of inequality \eqref{eq:gen2ineq}, and thus also of the upper bound \eqref{eq:gen2upper}.


\section*{Acknowledgment}
This work was supported in part by the Croatian Science Foundation under the project HRZZ-IP-2022-10-5116 (FANAP).


\bibliography{Gowers_on_the_cube}{}
\bibliographystyle{plainurl}

\end{document}